\documentclass[12pt,twoside,reqno]{amsart}
\usepackage[colorlinks=true,citecolor=blue]{hyperref}
\usepackage{mathptmx, amsmath, amssymb, amsfonts, amsthm, mathptmx, enumerate, color,mathrsfs}
\usepackage{graphicx}

\usepackage{multirow}
\usepackage{epstopdf}
\usepackage{multicol}
\usepackage{algorithm}
\usepackage{algorithmic}
\usepackage{epstopdf}

\newtheorem{theorem}{Theorem}[section]
\newtheorem{lemma}[theorem]{Lemma}
\newtheorem{proposition}[theorem]{Proposition}
\newtheorem{corollary}[theorem]{Corollary}

\theoremstyle{definition}
\newtheorem{definition}[theorem]{Definition}

\newtheorem{example}[theorem]{Example}

\numberwithin{equation}{section}

\usepackage{cite}\usepackage{hyperref}
\usepackage{amsmath,amssymb,amsfonts}
\usepackage{algorithmic}
\usepackage{graphicx}
\usepackage{textcomp}
\usepackage{amsmath,amssymb,bm,bbm,mathrsfs,amscd}
\usepackage{calc}
\usepackage{color}
\usepackage{amsthm}
\usepackage{dsfont}
\usepackage{graphicx}
\usepackage{epstopdf}
\usepackage{epsfig}
\usepackage{tikz}
\usepackage{comment}
\usetikzlibrary{patterns}
\usepackage{amsmath}
\usepackage{amsfonts}
\usepackage{amssymb}
\usepackage{rotating}
\usepackage{mathtools}
\usepackage{color}
\usepackage{subfig}
\usepackage{enumerate,pgfplots}

\newcommand{\ba}{\begin{array}}
\newcommand{\ea}{\end{array}}
\newcommand{\be}{\begin{equation}}
\newcommand{\ee}{\end{equation}}

\newcommand{\mc}{\mathcal}
\newcommand{\0}{\mathbf{0}}
\newcommand{\1}{\mathbf{1}}
\newcommand{\x}{\mathbf{x}}

\renewcommand{\S}{\mathcal{S} }

\newcommand{\R}{\mathbb{R}}
\newcommand{\y}{\mathbf{y}}
\newcommand{\infl}{h}
\newcommand{\Infl}{\mathbf{\infl}}
\newcommand{\thre}{\theta}
\newcommand{\threabs}{\theta}
\newcommand{\Thre}{\mathbf{\theta}}

\newcommand{\we}{W^*}
\DeclareMathOperator*{\argmin}{argmin}

\begin{document}
\setcounter{page}{1}

\vspace*{2.0cm}
\title[Intervention problems in the Linear Threshold Model]
{Intervention problems in the Linear Threshold Model: a general formulation and new results\\[7pt]
{\small This work is dedicated to the memory of Sanjoy K. Mitter.}}
\author[G.~Como, S.~Durand, F.~Fagnani]{ Giacomo Como$^{1}$, Stephane Durand$^{2}$, Fabio Fagnani$^{1}$}
\maketitle
\vspace*{-0.6cm}

\begin{center}
{\footnotesize

$^1$Department of Mathematical Sciences ``G.L.~Lagrange'', Politecnico di Torino, Italy

$^2$LISN,  Universit\'e Paris Saclay, France.  

}\end{center}

\vskip 4mm {\footnotesize \noindent {\bf Abstract.}
We study an optimal intervention problem for  linear threshold models. This is a popular class of dynamical network systems whereby a number of agents, identified with the nodes of a graph, strategically change their binary action ($0$ or $1$) according to a threshold rule. Specifically, an agent adopts action $1$ if and only if the fraction of its neighbors in the interaction graph that do so is greater than or equal to a prescribed threshold. Assuming that a planner can modify the agents' thresholds at a cost equal to the aggregate threshold increase, we study the minimum intervention cost needed to ensure global convergence to the all-$1$ configuration. Our main contribution is the introduction of a new graph-theoretic quantity, called oriented path number,  that is the minimum number of disjoint paths needed to cover the graph that can be oriented to form a directed acyclic graph.  When thresholds are all equal to $1/2$, the optimal cost is shown to coincide with  the oriented path number, whereas, in the general case, it turns out to be the main ingredient of a bound on the optimal intervention cost.


  }

\renewcommand{\thefootnote}{}
\footnotetext{ 
E-mail addresses: giacomo.como@polito.it (G.~Como), stephane.durand@universite-paris-saclay.fr (S.~Durand), fabio.fagnani@polito.it (F.~Fagnani).

}

\section{Introduction}


A prominent feature of network systems is the possible emergence of contagion evolutions in response to localized external shocks. 
The core of this phenomenon lies in the interplay between the nature of the single agents' dynamics and the topology of their interaction. A key issue in this regard is to find simple indices capable of measuring the propensity of a system to such behavior and, in a complementary way, of measuring its resilience. 

Considerable progress has been made for models with pairwise interactions such as, e.g., epidemics and gossip models. For many of them, a reasonably complete theory is available \cite{DM:2010}, which shows how the extent and the speed of a contagion spreading is essentially related to the Cheeger constant of the underlying interaction network. In these models, a basic straightforward fact is that the addition of connectivity in the graph always increases the propensity for contagion to spread.

Pairwise interactions, however, are often inadequate for modelling strategic behaviour in network systems. Examples include, for instance, the spread of innovations or new technologies in social and economic networks or contagion dynamics in financial networks. A popular model in such contexts is the Linear Threshold Model (LTM) of cascades first introduced by Granovetter for fully mixed systems \cite{mG:1978} and later extended in various directions  \cite{sM:2000, djW:2002, fVR:2007, EK:2010}. This model can be interpreted as the best response dynamics in a network game whereby agents choose strategically between two actions ($0$ and $1$) and their payoff is an increasing function of the number of their neighbors choosing the same action. 

Various studies of the LTM model \cite{sM:2000, hA:2010, mL:2012, Moharrami:2016:EC, Rossi.ea:2019, Ravazzi.ea:23, Messina.ea:24} have investigated topological conditions guaranteeing or preventing a full contagion (i.e., a configuration where all agents choose action $1$) starting from an initial condition of relatively few agents choosing action $1$. Most of these studies concern random networks of a specific type. A remarkable exception is \cite{sM:2000} that introduces the concept of cohesiveness, which, for general graphs, can in principle characterize the extent of a spreading phenomenon. Though simple and conceptually elegant, this index is unfortunately not easy to be computed and thus of limited help in the determination of the network properties mostly responsible for triggering cascade propagations.

Optimal seeding problems have been studied for LTM: most of them basically consist in forcing a certain number of agents to activate in order to finally maximize in some sense the spread of the activated agents. The optimization problems can be formulated as the problem of targeting a fixed number $K$ of agents so that, if activated, they yield the maximal possible expansion of the contagion \cite{KKT:2003, CYZ:2010,GLS:2011} . A complementary possibility is the minimal targeting problem that is the determination of the minimum number of agents that, if activated, will lead to a full cascade \cite{Durand:2022}. Both problems are known to be NP-complete and suboptimal algorithms have been proposed. The problem has also been tackled in the statistical physics literature \cite{Altarelli.et.al:2013,Guggiola.Semerjian:2015, Sen.et.al:2017}. 

In this paper, we propose and study a fundamental intervention problem for LTM models. We imagine that an external planner can intervene on the system by injecting energy in the various agents so as to modify their thresholds. The optimization problem we consider is to find the minimal-cost intervention needed to trigger a full cascade. The problem is fundamentally different from the classical seeding problems as the cost of the intervention depends on continuous parameters instead of just combinatorial ones. In many applications, modeling the intervention in this way turns out to be a more natural choice (e.g., an agent with a very high threshold of activation needs more external energy to push its activation). The   minimal-cost intervention problem for the LTM has been recently investigated in \cite{Cianfanelli.ea:2026}, where approximate solutions for large-scale networks have been studied. 

Our first contribution is to show that our optimal problem can be equivalently reformulated as a combinatorial optimization problem on the group of permutations of the set of agents. The idea behind this is that each permutation describes a possible order of activation of the agents for which the minimum intervention cost can be directly computed. In this way, the problem becomes finding the permutation that minimizes the cost. While, not too surprisingly, the derived combinatorial optimization problem is NP-complete (the proof will be proposed elsewhere), fundamental insight can be gained on the problem by exploiting the network structure.

\section{The model}\label{sec:model}
The linear threshold model (LTM) is one of the most studied models of cascading propagation in networks. It was originally proposed in \cite{mG:1978} for fully mixed populations and since then 
many similar models have been studied. 
The core of the model and of most of its generalizations can be described as follows.
Consider a directed weighted  graph $\mc G=(\mc V, \mc E, W)$ where $\mc V=\{1,\ldots,n\}$ is the set of nodes representing interacting agents, $\mc E\subseteq\mc  V\times \mc V$ is the set of links, and $W$ in $\R_+^{n\times n}$ is the weight matrix. 
We require that $W_{ij}>0$ if and only if $(i,j)\in\mc E$
and denote by $w_i=\sum_j W_{ij}$  the weighted degree of node $i$ and let $\we=\frac{1}{2}\sum_iw_i$. In the case when $W_{ij}\in\{0,1\}$ for every $i,j$, we call the graph unweighted. In this case the matrix $W$ coincides with the adjacency matrix of the underlying graph $(\mc V, \mc E)$, $w_i$ is the degree of the node $i$. If $\mc G$ is a graph, let $\mc G^{\leftrightarrow}=(\mc V,\mc E^{\leftrightarrow})$ denote the undirected completion of $\mc G$ formally defining its set of links as 
$$\mc E^{\leftrightarrow}=\{(i,j)\in\mc V\times\mc V\,|\, \{(i,j), (j,i)\}\cap\mc E\neq\emptyset\}\,.$$


We identify the nodes of the graph with agents and define a class of discrete-time dynamics where each agent chooses between action $0$ and $1$ according to the following mechanism. Consider threshold maps $f_{\alpha}:\R_+\to [0,1]$
$$f_{\alpha}(x)=\left\{\begin{array}{ll} 0\quad &\hbox{if}\; x<\alpha\\ 1\quad &\hbox{if}\; x\geq \alpha\,, \end{array}\right.$$
and consider the following update rule:
\be\label{LTE1}x_i(t+1)=f_{\thre_iw_i}\left(\sum\nolimits_{j}W_{ij}x_j(t)\right)\qquad \forall i=1,\ldots,n\,.\ee
The value $\thre_i$ in $[0,1]$ is called the \emph{threshold} of node $i$. In words, at every time instant, an agent adopts action $1$ if and only if the total weight of links connecting node $i$ to neighbors already adopting action $1$ at the previous time exceeds or equals the value $\thre_iw_i$. We can compactly represent (\ref{LTE1}) as a dynamical system on $\{0,1\}^{n}$:
\be\label{LTE2}\x(t+1)=F(\x(t))\ee
An \emph{equilibrium point} of (\ref{LTE2}) is any $z$ in $\{0,1\}^{n}$ such that $F(z)=z$.

In this paper, we consider an intervention problem where an external planner can modify the nodes' thresholds. Formally, given a vector $\Infl$ in $\R_+^{n}$ (to be referred to as the \emph{intervention vector}), we consider the modified dynamics 
\be\label{LTE1int}x_i(t+1)=f_{\thre_iw_i-h_i}\left(\sum\nolimits_{j}W_{ij}x_j(t)\right)\,,\ee
compactly represented as
\be\label{LTE2int}x(t+1)=F_{\Infl}(x(t))\,.\ee
Consider the related evolution $x(t)=F^t_{\Infl}(x(0))$ starting from a generic initial condition $x(0)$ and define an intervention $h$ as  \emph{successful} when the evolution converges to the all-$1$ vector, from every initial condition. Formally, let the set of successful interventions be
\be\mc H=\left\{\Infl\in \R_+^{n}\,|\, \lim\limits_{t\to +\infty}\x(t)=\1,\;\forall \x(0)\in\{0,1\}^{n}\right\}\,.\ee
\begin{proposition}\label{prop:reprH} The following facts hold:
\begin{enumerate}
\item[(a)] For $\x(0)=0$, the limit $l= \lim\limits_{t\to +\infty}\x(t)$ exists and $x(t)=l$ for $t\geq n$;
\item[(b)] $\mc H=\left\{\Infl\in \R_+^{n}\,|\, \x(n)=\1\;\hbox{for}\, \x(0)=0\right\}$.
\end{enumerate}
\end{proposition}
\begin{proof}
Consider the standard partial order on $\R^{n}$ whereby, for two vectors $\x$ and $\y$ in $\R^{n}$, we write that $\x\leq\y$ if $x_i\leq y_i$ for every $1\le i\le n$ and notice that $F_{\Infl}$ is monotonically nondecreasing with respect to such order. This implies that, since $F_{\Infl}(\0)\geq \0$, then $\x(t)=F^t_{\Infl}(\0)$ is monotonically nondecreasing and thus converges to a limit. Moreover, since every entry of $x(t)$ changes value at most once, by time $t=n$, we have reached a point for which $F^t_{\Infl}(\0)=F^{t+1}_{\Infl}(\0)$. From that point on, the sequence is stationary. This proves point (a).

Since $F_{\Infl}$ is monotonically nondecreasing, we have that
$$F^t_{\Infl}(\x(0))\geq F^t_{\Infl}(\0)\quad\forall t,\,\forall \x(0)$$
This yields point (b).
\end{proof}

We assume that every agent $i$ is equipped with a non-decreasing lower semicontinuous cost function $C_i:\R_+\to\R_+$: we interpret $C_i(\infl_i)$ as the cost incurred by the external planner to lower the threshold of agent $i$ of a value $\infl_i$. We gather all cost functions in a vector denoted by $\mathbf C$.
We formulate the following optimization problem.

\be\label{minimum-interv} C^*(\mc G,\thre, \mathbf C)=\min\left\{\sum\nolimits_i C_i(\infl_i)\,|\, \Infl\in\mc H\right\}\,.\ee
From representation (b) in Proposition \ref{prop:reprH} and the fact that $\Infl\mapsto F^n_{\Infl}(0) $ is a continuous map, we have that $\mc H$ is a compact set. This implies that the minimum above always exists.

We refer to the triple $(\mc G, \thre, \mathbf C)$ as to the linear threshold model (LTM) over $\mc G$ with threshold vector $\thre$ and vector of cost functions $\mathbf C$ and to the optimal value $C^*(\mc G,\thre, \mathbf C)$ as to the \emph{least activation cost} of the LTM $(\mc G,\thre, \mathbf C)$.

A special choice of the cost functions leads to a model largely analyzed in the literature. This is the case in which cost functions have the form $C_i(x)=c_i\1_{x>0}$, that is any nonzero intervention on an agent has the same cost. In this case, we can restrict to interventions $\Infl$ where $\infl_i\in\{0, \theta_i\}$ for the minimization problem \eqref{minimum-interv}. Such an intervention $\Infl$ can be identified with the subset $\mc W=\{i\,|\, \infl_i>0\}$ of agents on which the intervention is active. $\mc W$ is called the \emph{target} set and problem \eqref{minimum-interv} is called the \emph{target set selection problem}  (when $c_i=c$ for every $i$) or, when this homogeneity assumption does not hold, the \emph{weighted target set selection problem}. This problem has been introduced and studied in \cite{Chen:2009, Chen:2013} where its NP-completeness was proven. In \cite{Durand:2022}, for a more general 
model encompassing the LTM, a low complexity algorithm for an approximate solution of the weighted target set selection problem was proposed. This approach has then been extended to non-binary super-modular games in \cite{Messina.ea:23}. 

Another special case analyzed in the literature is when the cost functions $C_i(x)=c_ix$ are linear. In \cite{Cordasco.ea:2015, Gunnec.ea:2019} the problem has been introduced under the name of \emph{least cost influence problem} LCIS. Also in this case, the problem is shown to be NP-complete. Further complexity results are proven showing NP-completeness on special families of graphs (e.g., bipartite graphs, trees) and also establishing hardness-of-approximation results.

In next section, we show how the general problem \eqref{minimum-interv} can be reformulated as a pure combinatorial optimization problem, for every choice of the cost functions. Thereafter, we focus on the least cost influence problem: we present new bounds and a number of examples where an analytical solution can be found.

\section{A combinatorial reformulation of the problem}
In this section, we reformulate the least cost influence problem (\ref{minimum-interv}) as an optimization problem on the group of permutations of the node set $\mc V$. This allows us to introduce a number of concepts and energy interpretations that turn out to play a crucial role in the rest of our analysis.

Inspired by typical evolutions studied in game theory, where agents modify their behavior unilaterally, best-responding to the strategies of the other agents, we investigate here similar dynamics. Assume that all agents initially adopt action $0$ and then update their action one at a time and only once each, as the output of their threshold function $f_{\thre_iw_i-\infl_i}$ applied to the output of the previously activated agents. 
%
Let $\S_n$ be the group of permutations over the agent set $\mc V=\{1,\ldots,n\}$. The activation order is determined by the choice of a permutation $\sigma$ in $\S_{n}$. For $1\le t\le n$, the following expression is computed $$f_{\thre_{\sigma_t}w_{\sigma_t}-\infl_{\sigma_t}}(w_{\sigma_t}(\sigma))\,,$$
where
\be\label{weightsigma}w_{\sigma_t}(\sigma)=\sum\limits_{1\le s<t}W_{\sigma_t\sigma_s}\,,\ee
is the total weight of links connecting agent $\sigma_t$ to previously activated agents. 
Given the permutation $\sigma$, the minimum value of the intervention $\infl_{\sigma_t}$ that allows agent $\sigma_t$ to adopt action $1$ is given by
\be\label{intvectorsigma}\infl_{\sigma_t}({\sigma})=\thre_{\sigma_t}w_{\sigma_t}-w_{\sigma_t}(\sigma)\,.\ee
We assemble these values into a vector $\Infl({\sigma})$ in $\R_+^n$ to be referred to as the \emph{minimum intervention vector associated with $\sigma$}. Notice that some entries of $\Infl(\sigma)$ may be negative. We define the minimum cost associated with the permutation $\sigma$ as
\be\label{costsigma}
C(\sigma)=\sum\limits_{1\le i\le n}C_i([\infl_i(\sigma)]^+)=\sum\limits_{1\le i\le n}C_i([\thre_iw_i-w_i(\sigma)]^+)\ee
where $[x]^+=\max\{0,x\}$ denotes the positive part of a scalar $x$.  

The optimal activation cost computed in this way coincides with the one computed through the synchronous LTM dynamics defined by iterations of the map $F_{\Infl}$. 
\begin{proposition}\label{prop:optimal-cost} For every LTM $(\mc G,\theta)$, it holds that 
\be\label{optimalcost}
C^*(\mc G,\threabs)=\min\limits_{\sigma\in \S_{n}}C({\sigma})
\ee
\end{proposition}
\begin{proof}

We first show that for every permutation $\sigma$ in $\mc S_{n}$, we have that $\Infl(\sigma)\in\mc H$. First, we prove by induction that, for every $0\le t\leq n$, 
\be\label{monotone}x^{\Infl}(t)\geq \delta^{\sigma_1}+\cdots +  \delta^{\sigma_t}\,.\ee
For $t=0$, relation \eqref{monotone} trivially holds true. Moreover, if \eqref{monotone} holds true for some $t\ge0$, then, by definition of $\Infl(\sigma)$,
$$x^{\Infl}(t+1)_{\sigma_{t+1}}=f_{w_{\sigma_{t+1}}(\sigma)}((Wx^{\Infl}(t))_{\sigma_{t+1}})\geq f_{w_{\sigma_{t+1}}(\sigma)}(w_{\sigma_{t+1}}(\sigma))=1\,. $$
Since $x^{\Infl}(t+1)\geq x^{\Infl}(t)$, this yields relation (\ref{monotone}) for $t+1$. Relation (\ref{monotone}) for $t=n$  implies that $\Infl(\sigma)\in\mc H$. By definition, this yields $\leq$ in formula (\ref{optimalcost}).

To see that the converse inequality holds true as well, for every $\Infl$ in $\mc H$, we construct a permutation $\sigma^\Infl$ in $\mc S_n$ as follows. Start with $\x(0)=\0$ and, for $t=1,\ldots,n$ choose $i_t$ in $\mc V$ such that $x_{i_t}(t-1)=0$ and $f_{w_{i_t}\theta_{i_t}+h_{i_t}}(\sum_jW_{{i_t}j}x_j(t-1))=1$ and put $$\x(t)=\x(t-1)+\delta^{i_t}\,,\qquad \sigma^{\Infl}_t={i_t}\,.$$ 
Notice that such ${i_t}$ exists for every $1\le t\le n$ for otherwise $F_{\Infl}(\x(t-1))=\x(t-1)\ne\1$ and $\Infl\notin\mc H$. Moreover, by construction, $i_s\ne i_t$ for every $s\ne t$ so that $\sigma^{\Infl}$ is indeed a permutation in $\mc S_n$. Now, observe that $\Infl(\sigma^{\Infl})\le \Infl$, so that $C(\sigma^{\Infl})\le\sum_ih_i$. By the arbitrariness of $\Infl$ in $\mc H$ this implies that 
$$\min\limits_{\sigma\in \S_{n}}C({\sigma})\le\min\limits_{\Infl\in\mc H}C({\sigma}^{\Infl})\le\min\limits_{\Infl\in\mc H}\sum\nolimits_ih_i= C^*(\mc G,\threabs)\,,$$
thus completing the proof. 
\end{proof}
The set of permutations for which the cost is optimal is denoted $\mathcal S_{(\mc G, \threabs)}^*$ and we refer to each of them as to an \emph{optimal permutation activation}. 

%


\section{General bounds for the LCIS problem}
In the rest of this work, we restrict to the case when $C_i(x)=c_ix$. Since the combinatorial function to optimize $C(\sigma)$ can be rewritten as follows
$$C(\sigma)=\sum\limits_{1\le i\le n}c_i[\thre_iw_i-w_i(\sigma)]^+=\sum\limits_{1\le i\le n}[\thre_ic_iw_i-c_iw_i(\sigma)]^+\,,$$
an equivalent optimization problem can be obtained by changing $c$ to $\1$ and $W$ to $[c]W$. Therefore, from now on we assume that $c_i=1$ for every $i$ and thus focus on the combinatorial optimization problem
\be\label{comb-opt} \min\limits_{\sigma \in \mc S_{n}} \sum\limits_{1\le i\le n}c_i[\thre_iw_i-w_i(\sigma)]^+\,.\ee

{The optimization problem admits the following energy interpretation.
 For an agent $i$ to be activated, a minimum energy $\threabs_i$ is needed as input. When a specific permutation $\sigma$ is chosen, an agent $i$, at the moment of its activation, receives an amount of energy $w_i(\sigma)$ from the other previously activated agents and what is left $(\threabs_iw_i-w_i(\sigma))^+$ needs to be provided by the external intervention.  
A strategy, generally suboptimal, is that of maximizing the energy received by previously activated nodes, namely, to consider
\be\label{we-def} \we=\max\limits_{\sigma \in \mathcal S_n}\sum\limits_{i=1}^n w_i(\sigma)\,.\ee
Proposition \ref{prop:optimal-cost} yields the following simple bounds
\be\label{sumhsigma2}\left(\sum\limits_{i\in\mc V}\threabs_iw_i-\we\right)^+\leq C^*(\mc G, \threabs)\leq \sum\limits_{i\in\mc V}\threabs_iw_i\,.\ee
Both bounds can be achieved. 
Notice that the upper bound corresponds to
%
the cost of an intervention when $w_i(\sigma)=0$ for every $i$ and is trivially achieved when the agents are totally isolated $W_{ij}=0$ for every $i\neq j$. The lower bound represents instead the cost of an intervention when $w_i(\sigma)\leq \threabs_iw_i$ for every agent $i$ for a permutation maximizing \eqref{we-def}. This happens for instance in the extreme case where $\theta_i=1$ for every $i$.
Both bounds can be improved. To this aim, we consider the decomposition of the graph $\mc G$ into its connected components, and we denote by $\mc G_k=(\mc S_k, \mc E_k)$ for $k=1,\dots ,q$ the sink ones. We put $q_{\mc G}=q$ and we define 
\be  i^*_k\in\argmin\limits_{i\in \mc S_k} \thre_iw_i,\quad k=1,\dots , q,\qquad \rho=\sum\limits_{k=1}^q\thre_{i^*_k}w_{i^*_k},\qquad \mc V^*=\{i^*_k\,|\, k=1,\dots ,q\}\ee
We now introduce the two values
\be\label{indices} C^{max}(\mc G,\threabs)=\rho+\sum\limits_{i\in\mc V\setminus\mc V^*}(\thre_iw_i-w)^+, \quad C^{min}(\mc G,\threabs)=\rho+\left(\sum\limits_{i\in\mc V\setminus \mc V^*}\threabs_iw_i-\we\right)^+\ee
and the following set of permutations.
\begin{definition} Given a graph $\mc G$, a permutation $\sigma$ is called $\mc G$-\emph{greedy} if 
\be\label{sigma-neigh}  \sigma_t\in\bigcup\limits_{s<t}N^-_{\sigma_s}\ee
but for the first $q_{\mc G}$ values of $t$.
\end{definition}
\begin{proposition} Given a graph $\mc G$, the following facts hold:
\begin{enumerate}
\item $\mc G$-{greedy} permutations exist
\item For every $\mc G$-{greedy} permutation $\sigma$, we have that $C(\sigma)\leq C^{max}(\mc G,\threabs)$.
\end{enumerate}
\end{proposition}
\begin{proof}
1. To prove the existence of $\mc G$-{greedy} permutations, it is sufficient to choose arbitrarily $\sigma_t\in\mc S_t$ for $t=1,\dots , q$ and then, successively, to choose the remaining elements imposing \eqref{sigma-neigh}. The fact that this is possible  is equivalent to the assumption that no connected component of $\mc G$ different from $\mc G_k$ for $k=1,\dots ,q$ is a sink component. 

2. If $\sigma$ is a $\mc G$-{greedy} permutation, then $w_t(\sigma)\geq w$ for every $t>q$ and thus
$$\begin{array}{rcl}C(\sigma)&=&\sum\limits_{t=1}^n(\thre_{\sigma_t}w_{\sigma_t}-w_t(\sigma))^+\\&=&\rho+\sum\limits_{t> q}(\thre_{\sigma_t}w_{\sigma_t}-w_t(\sigma))^+\\
&\leq &\rho+\sum\limits_{t>q}(\thre_{\sigma_t}w_{\sigma_t}-w)^+\\ &=&C^{max}(\mc G,\threabs)\end{array}$$
\end{proof}

We denote by $\mc S_n^g$ the subset of the $\mc G$-greedy permutations.
The following result holds.
\begin{corollary}[Basic bound] 
\label{cor:conservation} For the LTM $(\mc G,\threabs)$ on a graph $\mc G=(\mc V, \mc E, W)$ of order $n$ with $\thre_i>0$ for every $i$: 
\be\label{basic-bound}C^{min}(\mc G,\threabs)\leq C^*(\mc G, \threabs)\leq C^{max}(\mc G,\threabs)\,.\ee
\end{corollary}
 
\begin{proof} 
We only need to prove the lower bound. This  follows from the observation that for any permutation $\sigma$, the external planner must pay in full for at least $q$ times, each time that it chooses for the first time a node in $\mc S_k$ for $k=1,\dots , q$. 
Precisely, define
$$t^*_k=min\{t\,|\, \sigma_t\in\mc S_k\},\quad \mc T=\{t^*_1,\dots , t^*_q\},\quad \mc T^*=\{1,\dots , q\}\setminus\mc T$$
The following inequalities hold.
$$\begin{array}{rcl}C(\sigma)&=&\sum\limits_{t\in\mc T}\threabs_{\sigma_{t}}w_{\sigma_{t}}+\sum\limits_{t\in\mc T^*}(\threabs_{\sigma_t}w_{\sigma_t}-w_t(\sigma))^+\\ 
&\geq&\sum\limits_{t\in\mc T}\threabs_{\sigma_{t}}w_{\sigma_{t}}+\left(\sum\limits_{t\in\mc T^*}\threabs_{\sigma_t}w_{\sigma_t}- \sum\limits_{t\in\mc T^*}w_t(\sigma))^+\right)\\ 
\end{array}$$
Since, by construction, $\sum\limits_{t\in\mc T}\threabs_{\sigma_{t}}w_{\sigma_{t}}\geq \rho$ and 
$$\sum\limits_{t\in\mc T}\threabs_{\sigma_{t}}w_{\sigma_{t}}+\sum\limits_{t\in\mc T^*}\threabs_{\sigma_t}w_{\sigma_t}=\sum\limits_{i\in\mc V}\threabs_iw_i=\rho+\sum\limits_{i\in\mc V\setminus \mc V^*}\threabs_iw_i$$
the left-hand side inequality in \eqref{basic-bound} follows.
\end{proof}

There is a case where the bounds coincide and equal the least cost intervention. 
\begin{proposition}\label{prop:special} Let $\mc G=(\mc V, \mc E, W)$ be an unweighted graph of order $n$ and let $\threabs\in (0,1]^n$ be a threshold vector such that $\thre_iw_i\leq 1$ for every $i$. Then,
\be\label{conn}C^{min}(\mc G,\threabs)=C^{max}(\mc G,\threabs)= C^*(\mc G, \threabs)=\rho\ee
Moreover, any $\mc G$-greedy permutation is optimal.
%
\end{proposition}
\begin{proof} 
Since $\we\geq n-q$ on every unweighted graph with $q$ sink connected components, we obtain that $C^{min}(\mc G,\threabs)=C^{max}(\mc G,\threabs)=\rho$. Moreover, if $\sigma$ is a $\mc G$-greedy permutation, then $w_t(\sigma)\geq 1$ for every $t>q$ and formula (\ref{costsigma}) yields $C(\sigma)=\rho$.
\end{proof}

\begin{example}[Directed Acyclic Graphs]\label{ex:DAG} Consider any DAG $\mc G$ over the set of nodes $\mc V=\{1,2,\dots , n\}$ and any threshold vector $\theta$. Each of the  $q=q_{\mc G}$ sink components are in this case composed of a single element: denote them by $i^*_1, \dots , i^*_q$. As $w_{i^*_k}=0$ for $k=1,\dots, q$, we obtain that $\rho=0$. 
Since $\mc G$ is a DAG,  we can find a total order of the nodes $k_1<k_2<\cdots <k_n$ such that for every $i, j\in\{1,\dots , n\}$, $W_{k_ik_j}>0$ implies $i>j$. If we now choose $\sigma_t=k_t$ for every $t$, we have that $w_{k_t}(\sigma)=w_{k_t}$ for every $t=1,\dots ,n$. Consequently, $C(\sigma)=\rho=0$. Such a permutation is thus optimal and $C^*(\mc G, \threabs)=0$.
\end{example}%

\section{LCIS over undirected graphs}
In this section, we analyze in greater detail the case when the underlying graph $\mc G$ is undirected. A striking result is that in this case the maximization problem in \eqref{we-def} is trivial as all permutations $\sigma$ reach the same value. Precisely, the following conservation law holds.

\begin{proposition}\label{lemma:equalize} For every undirected weighted graph $\mc G=(\mc V, \mc E, W)$ of order $n$,
$$\we = \frac{1}{2}\sum\limits_{i\in\mc V}w_i$$
Moreover, 
for every permutation $\sigma$ in $\mc S_{n}$,  
\be\label{equalize}\sum\nolimits_{i} w_i(\sigma)=\we \qquad  \sum\nolimits_{i} h_i(\sigma)=\sum\nolimits_{i}\thre_iw_i- \we
\,.\ee
\end{proposition}
\begin{proof}
From (\ref{weightsigma}) we obtain
$$\sum\nolimits_{i} w_i(\sigma)=\sum\limits_{1\le t\le n} w_{\sigma_t}(\sigma)
=\sum\limits_{1\le s<t\le n}W_{\sigma_t\sigma_s}
=\frac{1}{2}\sum\limits_{1\le t\le n}\sum\limits_{1\le s\le n}W_{\sigma_t\sigma_s}=\frac{1}{2}\sum\nolimits_{i}w_i\,,$$
This proves the first formula. The second one is a direct consequence of \eqref{intvectorsigma} and previous formula.
\end{proof}

Symmetry of the interaction among the agents (following from the assumption that $W$ is symmetric) yields an interesting reversibility with respect to a transformation in which agents exchange their thresholds $\thre_i$ with the complementary ones $1-\thre_i$. 
%
%
%
Precisely, for every permutation $\sigma$ in $\mc S_{n}$, we define the \emph{reversed permutation} $\bar\sigma$ as
$$\bar\sigma_i= \sigma_{n-i+1}\,,\qquad 1\le i\le n\,.$$
The cost of the permutation $\sigma$ for the LTM $(\mc G, \Thre)$ turns out to be intimately connected to the cost of $\bar \sigma$ for the reversed LTM $(\mc G, \1-\Thre)$. This is clearly stated below.
\begin{lemma}\label{lemma:reversed} Consider an LTM $(\mc G, \Thre)$ and indicate with $C$ and $\bar C$, respectively, the cost functions related to $(\mc G, \Thre)$ and to the reversed LTM $(\mc G, \1-\Thre)$. Then, for every permutation
$\sigma$ in $\S_{n}$, 
\be\label{dualcost}\bar C({\bar \sigma})=C({\sigma})-\sum\limits_{i\in\mc V}\threabs_{i}w_i+\we
\ee
\end{lemma}
\begin{proof}
The key point is to observe that the definition of the reversed permutation $\bar\sigma$ yields $w_i(\bar\sigma)=w_i-w_i(\sigma)$. This yields
$$\begin{array}{rcl}\bar C({\bar\sigma})&=&\sum\nolimits_{i} [(1-\theta_i)w_i-w_i(\bar\sigma)]^+\\ [3pt]
&=&\sum\nolimits_{i} [\theta_iw_i-w_i(\sigma)]^-\\[3pt]
&=&\sum\nolimits_{i} [\theta_iw_i-w_i(\sigma)]^+- \sum\nolimits_{i} (\theta_iw_i-w_i(\sigma))\\[3pt]
&=&
C(\sigma)-(\sum\nolimits_{i}\threabs_{i}w_i-\we)\,,
\end{array}$$
thus proving the result.
\end{proof}

Previous result yields a direct connection between the optimization problems relative to the two LTM $(\mc G,\Thre)$ and $(\mc G,\1-\Thre)$.

\begin{corollary}[reversal]
\label{cor:rev} Consider an LTM $(\mc G, \Thre)$. Then,
$$C^*(\mc G,\Thre)=C^*(\mc G,\1-\Thre)-\sum\limits_{i\in\mc V}\threabs_{i}w_i+\we$$
and
 $$\sigma\in \mathcal S_{(\mc G,\Thre)}^*\;\Leftrightarrow\; \bar\sigma\in \mathcal S_{(\mc G,\1-\Thre)}^*$$
\end{corollary}
\begin{proof}
Direct consequence of Lemma \ref{lemma:reversed}.
\end{proof}

}

\subsection{Some basic examples}
In this subsection, we carry on some exact computation for some simple basic topologies. Throughout this section, we stick to undirected graphs $\mc G=(\mc V,\mc E, W)$ where $W_{ij}\in\{0,1\}$ for every $i,j$. In this case we refer to $\mc G$ as to an unweighted graph: in this case $w_i$ coincides with the degree of node $i$, namely the cardinality of the neighborhood $N_i$ of node $i$. 

\begin{example}[Line and Ring graphs]\label{ex:ring0}\label{ex:ring1} Consider the line graph $L_n$ or the ring graph $R_{n}$ over the set $\mc V=\{1,2,\dots , n\}$. 
Assuming that nodes are numbered in the usual way, as in the top part of Figure \ref{fig:line},
the cost of the identity permutation $\sigma=id$ for any LTM over $L_n$ or $R_n$ can be computed as follows
\be\label{trivial-line}
C({id})=w_1\thre_1+\sum\limits_{i=2}^{n-1}[2\thre_i-1]^+\ee
Notice that $w_1=1$ in the case of a line while $w_1=2$ in the case of a ring. 

In the special case when $\thre_i=\theta\leq 1/2$ for all $i$, we have that $C({id})=w_1\thre_1=\min\{w_i\thre_i\}$ and Proposition \ref{prop:special} implies that $id$ is an optimal permutation.
%
%
The mechanism underlying this result is the following: once the first agent becomes active, the topological structure of the line together with the fact that the threshold is below $1/2$ leads to a cascade effect that activates the entire line without requiring further intervention energy.

If thresholds are different, the identity permutation is in general not optimal. There are two special cases where an exact solution can nevertheless be obtained. This occurs when thresholds are either all below or all above $1/2$.

In the first case, $\thre_i\leq 1/2$ for every $i$, by Proposition \ref{prop:special}, we obtain that any permutation satisfying (\ref{sigma-neigh}) is optimal. In particular $id$ is optimal if and only if $w_1\thre_1=\min\{w_i\thre_i\}$.
%

We now use Corollary \ref{cor:rev} to solve the case when $\theta_i\geq 1/2$ for all $i$. The optimal permutation is given by any $\sigma^*$ such that $\bar{\sigma^*}$ is the optimal permutation for the model with complementary thresholds $1-\thre_i$. This implies that $\sigma^*$ is characterized as follows.
$$\sigma^*_n=\argmin\limits_{i\in\mc V}((1-\thre_i)w_i),\qquad \sigma^*_{k}\in 
\left(\bigcup\limits_{h\geq k+1}N_{\sigma^*_{h}}\right)\setminus\{\sigma^*_{k+1},\dots , \sigma^*_{n}\}
$$
Notice that $\sigma^*_1$ is not necessarily an agent minimizing $\thre_iw_i$ and $\sigma^*$ not necessarily is a $\mc G$-greedy permutation. Using formula (\ref{dualcost}), the cost of this permutation is given by
$$C(\sigma^*)=\min_{i\in\mc V}(1-\thre_i)w_i-2\sum_{i\in\mc V}\left(1-\theta_i-\frac{1}{2}\right)
=\min_{i\in\mc V}(1-\thre_i)w_i+\sum_{i\in\mc V}\left(2\theta_i-\1\right)\,.$$

\end{example} 

\begin{figure}
\begin{center}
\includegraphics[height=0.4cm]{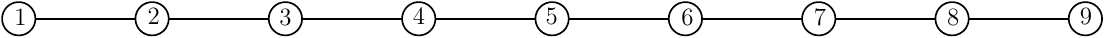} 
\end{center}
\caption{The line graph with $n=9$ }
\label{fig:line}
\end{figure}

\begin{example}[Trees] Let $\mc G=(\mc V, \mc E, W)$ be an unweighted tree of order $n$ and let $\Thre$ be any threshold vector. Assume that $\thre_iw_i\geq 1$ for every $i$. 
Notice that  $\we=n-1$ in this case. Using definitions (\ref{indices}) we obtain that $C^{min}(\mc G,\threabs)=C^{max}(\mc G,\threabs)$ and any $\mc G$-greedy permutation $\sigma$ is thus optimal and we obtain
$$C(\sigma)=\thre_{\sigma_1}w_{\sigma_1}+\sum\limits_{1\le t\le n}(\thre_{\sigma_t}w_{\sigma_t}-1)=\sum\limits_{i\in\mc V}\thre_iw_i- (n-1)$$
\end{example}
%
%

\begin{example}[Complete graph] Consider the complete graph $K_{n}$. Let $\Thre$ be any threshold vector and consider any permutation $\sigma^*$ in $\mc S_{n}$ such that thresholds are in non-decreasing order, namely,
\be\label{permutation-order}h<k\;\Rightarrow\theta_{\sigma^*_h}\leq \theta_{\sigma^*_k}\,.\ee
Then $\sigma^*$ is an optimal permutation. Indeed, if $\sigma\in\mc S_{n}$ is any permutation, consider the quantity
\be\label{absolute}\sum\limits_{i\in\mc V}|\theta_iw_i-w_i(\sigma)|=\sum\limits_{1\le k\le n}|\theta_{\sigma_k}(n-1)-(k-1)|\,,\ee
and notice that if, for some $k$, it happens that $\theta_{\sigma_k}>\theta_{\sigma_{k+1}}$, then it holds
$$ |\theta_{\sigma_k}(n-1)-(k-1)|+|\theta_{\sigma_{k+1}}(n-1)-k|>
|\theta_{\sigma_{k+1}}(n-1)-(k-1)|+|\theta_{\sigma_{k}}(n-1)-k|\,.$$
In other words, if $\sigma$ does not satisfy condition (\ref{permutation-order}), then there exists another permutation that allows to lower the value (\ref{absolute}).
Since the cost can also be written as 
$$C(\sigma)=\sum_{i\in\mc V}[\theta_iw_i-w_i(\sigma)]^+=\frac{1}{2}\left(\sum\limits_{i\in\mc V}\theta_iw_i-\we\right)+\frac{1}{2}\sum_{i\in\mc V}|\theta_iw_i-w_i(\sigma)|\,,$$
this implies that if a permutation does not satisfy  (\ref{permutation-order}), it cannot be optimal. As optimal permutations surely exist, they must necessarily be of the form above. Notice moreover that, whenever $\theta_{\sigma_k}=\theta_{\sigma_{k+1}}$, inverting the two agents would not modify the cost. This says that all permutations for which the thresholds are in non-decreasing order in the sense of (\ref{permutation-order}) are optimal.
\end{example}

%

\section{Conclusion}
We have introduced and studied a new and natural optimal intervention problem for the linear threshold model. Future research directions include formally proving the complexity of the resulting combinatorial optimization problem and developing approximation algorithms for its solution. 


\noindent

\end{document}